\documentclass[letterpaper,12pt,titlepage,oneside,final]{article}

\usepackage{amsmath,amssymb,amstext} 
\usepackage[pdftex]{graphicx} 
\usepackage[pdftex,pagebackref=false]{hyperref} 
\hypersetup{
    plainpages=false,       % needed if Roman numbers in frontpages
    unicode=false,          % non-Latin characters in Acrobat’s bookmarks
    pdftoolbar=true,        % show Acrobat’s toolbar?
    pdfmenubar=true,        % show Acrobat’s menu?
    pdffitwindow=false,     % window fit to page when opened
    pdfstartview={FitH},    % fits the width of the page to 
    pdfnewwindow=true,      % links in new window
    colorlinks=true,        % false: boxed links; true: colored links
    linkcolor=blue,         % color of internal links
    citecolor=green,        % color of links to bibliography
    filecolor=magenta,      % color of file links
    urlcolor=cyan           % color of external links
}

\let\origdoublepage\cleardoublepage
\newcommand{\clearemptydoublepage}{
  \clearpage{\pagestyle{empty}\origdoublepage}}
\let\cleardoublepage\clearemptydoublepage

\newcommand{\iprod}{\mathbin{\lrcorner}}

\newcommand{\vol}{\operatorname{vol}}
\DeclareMathOperator{\Tr}{Tr}
\usepackage{xcolor}
\usepackage{epsf}
\usepackage{amsmath,amscd}
\usepackage{amsmath, latexsym, amsfonts, amssymb, amsopn, amscd, mathtools, esvect, bigints, graphicx, float, indentfirst, inputenc, amssymb, pdfpages, multirow, multicol, physics, url,alltt, mathrsfs, amsthm, tikz-cd,tikz, pgfplots}
\pgfplotsset{compat=1.18}
\usepackage{import, xifthen, transparent}
\usepackage{etoolbox, adjustbox}
\usepackage{stmaryrd}
\usepackage{subfiles}
\usepackage{enumerate}
\usepackage[dvipsnames]{xcolor}
\usepackage{color,soul}
\renewcommand{\l}{\left(}
\renewcommand{\r}{\right)}

\newtheorem{theorem}{Theorem}[section]

\newtheorem{lemma}[theorem]{Lemma}
\newtheorem{Def}[theorem]{Definition}
\newtheorem{conjecture}[theorem]{Conjecture}
\newtheorem{remark}[theorem]{Remark}

\begin{document}

\pagestyle{plain}
\pagenumbering{arabic}
\setcounter{page}{1}

% --- CUSTOM TITLE & AUTHOR ---
\begin{center}
    \Huge
    {\bf The hypersymplectic flow descended from the $G_2$-Laplacian coflow} \\
    
    \vspace*{1.0cm}
    \Large
    Amanda Maria Petcu \\
    \vspace*{0.5 cm}
    \large
   April 20 2026 \\
    \vspace*{1.5cm}
\end{center}

% --- ABSTRACT ---

\begin{center}\textbf{Abstract}\end{center}

A conjecture of Simon Donaldson is that on a compact $4$-manifold $X^4$ one can flow from a hypersymplectic structure to a hyperk\"ahler structure while remaining in the same cohomology class. To this end the hypersymplectic flow was introduced by Fine--Yao. In this paper the notion of a positive triple on $X^4$ is used to describe a hypersymplectic and hyperk\"ahler structure. Given a closed positive triple one can define either a closed $G_2$ structure or a coclosed $G_2$ structure on $\mathbb{T}^3 \times X^4$. The coclosed $G_2$ structure is evolved under the modified $G_2$-Laplacian coflow. The coflow descends to a flow of the positive triple on $X^4$, which is again the Fine--Yao hypersymplectic flow. 
\vspace*{1.0cm}

% --- INTRODUCTION ---

%======================================================================

%======================================================================
\section{Introduction}

The notion of a hyperk\"ahler manifold  first arose in Berger's classification of possible holonomy groups in his paper \cite{MR79806}, they are known in the literature to be a class of manifolds with special holonomy. Restricting to dimension $4$, it was  Donaldson who posed the idea of loosening the conditions of a hyperk\"ahler structure to one where the triple of K\"ahler forms was only symplectic.  This is what we call in this paper a hypersymplectic structure. In \cite{donaldson2006twoformsfourmanifoldsellipticequations}, Donaldson conjectured that up to isotopy the only example of a hypersympletic manifold is a hyperk\"ahler manifold.  We state the conjecture below and refer the reader to \cite{fine2020reporthypersymplecticflow} for more details on the conjecture and how it arose in the literature.
\begin{conjecture}{(Donaldson's Conjecture \cite{donaldson2006twoformsfourmanifoldsellipticequations})}
Let $M$ be a compact oriented $4$-manifold. Let $\underline{\omega}$ be a hypersymplectic structure on $M$. There exists a one-parameter family of diffeomorphisms $F_t: M \to M$ connected to the identity, that take a hypersymplectic structure $\underline{\omega}(0) = \underline{\omega}$ to a hyperk\"ahler structure $\underline{\omega}(1) = F_1^* \underline{\omega} $, so that for $0 \leq t \leq 1$, $\underline{\omega}(t) = F_t^*\underline{\omega}$ remains in the same cohomology class. 
\end{conjecture}

In a particular setting, Fine and Yao have done work to prove this conjecture. In \cite{fine_2018} the authors define a closed $G_2$ structure on $M^7 =\mathbb{T}^3 \times X^4 $ where $X^4$ is compact and hypersymplectic.
Let $t^i$ be the flat coordinates on $\mathbb{T}^3$ and let $\underline{\omega}$ be a hypersymplectic structure on $X^4$. The  $3$-form $\varphi$ on $M^7$  is given by 
\begin{equation}\label{G2 form}
    \varphi = dt^{123} - dt^1 \wedge \omega_1 - dt^2 \wedge \omega_2 - dt^3 \wedge \omega_3.
\end{equation}
This form gives a closed $G_2$ structure on $M^7$ that is invariant under $\mathbb{T}^3$, see \cite{fine_2018}.
 Using the particular $\varphi$ from \eqref{G2 form}, in \cite{fine_2018} the authors computed that the metric $g_{\varphi}$, given by 
\begin{equation}
    g_{\varphi}(a,b) \mu_{\varphi} = \frac{1}{6}\l a \iprod \varphi \r \wedge \l b \iprod \varphi \r \wedge \varphi
\end{equation}
is a warped product of the metric on $X^4$ given by the hypersymplectic structure and the Euclidean metric on $\mathbb{T}^3$.  It is given as 
\begin{equation}\label{G2 metric}
    g_{\varphi} =  Q_{ij} dt^i \otimes dt^j + g_{\underline{\omega}},
\end{equation}
where $g_{\underline{\omega}}$ is a metric on $X^4$ coming from the hypersymplectic triple $\underline{\omega}$. It is defined in \eqref{metric}. They evolve the $G_2$ structure under the $G_2$-Laplacian flow to create a new geometric flow on $X^4$ that could deform the hypersymplectic structure to a hyperk\"ahler structure. This flow is known as the hypersymplectic flow and is given by 
\begin{equation}\label{Hypersymplectic flow General}
    \partial_t \underline{\omega} = d\l Q d^* \l Q^{-1} \underline{\omega}\r\r.
\end{equation}
Letting $Q^{ij}$ be the entries of $Q^{-1}$, we can write the hypersymplectic flow more explicitly on each form as 
 \begin{equation}\label{Hypersymplectic flow component}
     \frac{\partial \omega_i}{\partial t} = d \l Q_{ik} d^* \l Q^{kl} \omega_l \r \r.
 \end{equation}
Furthermore, in \cite{fine_2018} the authors show that any solution of the $G_2$-Laplacian flow that begins with initial condition \eqref{G2 form} remains $\mathbb{T}^3$ invariant for as long as it exists. Thus any solution of the $G_2$-Laplacian flow of the form \eqref{G2 form} remains of this form along the flow. In \cite{fine2020reporthypersymplecticflow} the authors show that when $Q$ is constant, $\underline{\omega}$ is a fixed point of the flow. In such cases there exists a hyperk\"ahler structure whose metric is equal (up to a positive constant) to the hypersymplectic metric. The hypersymplectic flow $\eqref{Hypersymplectic flow General}$ descended from the $G_2$-Laplacian flow therefore one can say that the flows are equivalent in this particular set up. Thus the hypersymplectic flow inherits the properties of short-time existence and uniqueness from the $G_2$-Laplacian flow. Furthermore, if the initial $G_2$ structure $\varphi_0$ is closed then under the $G_2$-Laplacian flow, the structure $\varphi(t)$ remains closed and belongs to the same cohomology class as $\varphi_0$ for all time $t$, see \cite{bryant2011laplacianflowclosedg2structures}, \cite{karigiannis2025bridgeslecturesflowsgeometric}. Thus the same can be said about the hypersymplectic flow. If $\underline{\omega}_0$ is the initial hypersymplectic structure then as we flow using $\eqref{Hypersymplectic flow General}$ $\underline{\omega}(t)$ remains closed and cohomologous to $\underline{\omega}_0$ for as long as it exists. In \cite{fine_2018} the authors are able to show long time existence of the flow provided the torsion tensor of the $G_2$ structure stays uniformly bounded as we deform the hypersymplectic structures. In \cite{fine2024convergencehypersymplecticflowt4} the authors prove Donaldson's conjecture explicitly on $\mathbb{T}^3 \times \mathbb{T}^4$, under the assumption that the hypersymplectic structure on $\mathbb{T}^4$ is $\mathbb{T}^3$-invariant for some $\mathbb{T}^3 \subset \mathbb{T}^4$. We refer the reader to \cite{fine_2018}, \cite{fine2020reporthypersymplecticflow} and \cite{fine2024convergencehypersymplecticflowt4} for more details. Presently \cite{fine_2018} and \cite{fine2024convergencehypersymplecticflowt4} are the only known convergence results for the hypersymplectic flow. 
\\
\\
In this paper we construct a coclosed $G_2$ structure on $\mathbb{T}^3 \times X^4$ where $X^4$ is compact and hypersymplectic. We flow this structure under the modified $G_2$-Laplacian coflow. To this end we introduce some background for the modified Laplacian coflow and refer the reader to \cite{grigorian2018flowscoclosedg2structures} and \cite{Grigorian_2013} for additional details. One can define a $G_2$ structure and its corresponding Riemannian metric using a positive $4$-form $\psi$ (up to a choice of orientation). So instead of flowing the $3$-form $\varphi$ in time, one might choose to flow the $4$-form in time, thus creating a different type of $G_2$ geometric flow. Using this idea, the Laplacian coflow was first introduced in  \cite{Karigiannis_2012}. This flow is given by 
 \begin{equation}\label{Laplacian Coflow}
     \frac{\partial \psi }{\partial t} = dd^* \psi
 \end{equation}
 on the cohomology class $[\psi] = \{\psi_0 + d\eta\}$ with initial conditions $\psi(0) = \psi_0$ and $d\psi_0 = 0$. This flow shares some properties with the Laplacian flow. In particular the solutions stay coclosed and remain in the same cohomology class for their entire existence, see \cite{Grigorian_2013}, \cite{grigorian2018flowscoclosedg2structures}. However the main difference between the flows is that it is still not known if the Laplacian coflow has short time existence. In \cite{bryant2011laplacianflowclosedg2structures}, the Laplacian flow starting from closed $G_2$ structures is weakly parabolic in the direction of closed $3$-forms and thus an application of DeTurk's trick transforms this flow to a strictly parabolic flow which then has short-time existence. However the same cannot be done for the Laplacian coflow. This flow is not weakly parabolic in the direction of closed $4$-forms and thus cannot be shown to be equivalent to some parabolic flow, see \cite{grigorian2018flowscoclosedg2structures}.  In the hopes of finding coclosed $G_2$-structures that can be deformed to torsion-free $G_2$ structures, Grigorian introduced the modified Laplacian coflow in \cite{Grigorian_2013}. This flow is given by 
 \begin{equation}\label{modified coflow}
     \frac{\partial \psi }{\partial t} = dd^* \psi + 2d\l \l A- \Tr \mathrm{T}_{\psi} \r *_{\psi}\psi\r
 \end{equation}
where $\Tr \mathrm{T}_{\psi}$ is the trace of the full torsion tensor $\mathrm{T}_{\psi}$ of the $G_2$-structure defined by $\psi$ and $A$ is a fixed constant. We note that we can pick $A=0$ and this is exactly what we do in Section \ref{Chaper Laplacian Co-flow}. It can be seen that when $A=0$ and $\Tr \mathrm{T}_{\psi} = 0$ the coflow and modified coflow agree. In \cite{Grigorian_2013}, when starting from a coclosed $G_2$-structure the author shows that the modified coflow \eqref{modified coflow} is weakly parabolic in the direction of closed forms and hence an application of DeTurk's trick can be used to show that the flow has short time existence.
The author also uses the Nash-Moser inverse function theorem to show that the modified coflow also has uniqueness. These are the same techniques used by Bryant-Xu for the Laplacian flow in \cite{bryant2011laplacianflowclosedg2structures}. We summarize this to say that the following initial value problem 
\begin{equation}\label{modified coflow IVP}
    \begin{split}
        \frac{\partial \psi }{\partial t} &= dd^* \psi + 2d\l \l A- \Tr \mathrm{T}_{\psi} \r *_{\psi}\psi\r, \\
        \psi (0) &= \psi_0, \\
        d \psi_0 &= 0,
    \end{split}
\end{equation}
on the space $\{\psi_0 + d \eta\}$ has short-time existence and uniqueness. By construction any solution $\psi(t)$ remains in the same cohomology class $[\psi_0]$ for as long as it exists, and thus any solution stays coclosed for all time, see \cite{Grigorian_2013}. 
\\
\\
Following analogously the ideas of Fine and Yao, we construct a coclosed $G_2$ structure on $\mathbb{T}^3 \times X^4$ and show that this $G_2$ structure is invariant under $\mathbb{T}^3$. We evolve this $G_2$ structure under the modified $G_2$-Laplacian coflow and find that its solutions remain $\mathbb{T}^3$ invariant for as long as the flow exists. Furthermore, we find that for this particular ansatz the modified Laplacian coflow descends to a flow of the hypersymplectic structure on ${X}^4$, which is the same flow that Fine-Yao discovered known as the hypersymplectic flow.

\noindent In Section \ref{Section Geometric Structures}, we begin by giving some background. We introduce a positive triple of $2$-forms on a $4$-manifold, which we use to describe a hypersymplectic structure and a hyperk\"ahler structure. In Section \ref{Chaper Laplacian Co-flow}, the main results of this paper are given. In Section \ref{Dual positive triple and a coclosed $G_2$ structure}, we explain how given a hypersymplectic triple we can either define a closed $G_2$ structure or a coclosed $G_2$ structure. We construct a coclosed $G_2$ structure on $\mathbb{T}^3 \times X^4$ where $X^4$ is compact. In Section \ref{The hypersymplectic flow from the Laplacian coflow}, we show that this $G_2$ structure is invariant under $\mathbb{T}^3$. We also show that in this setting the $G_2$-Laplacian coflow and modified coflow agree. We flow this structure under the $G_2$-Laplacian coflow and find that it descends to the same flow of hypersymplectic structures found by Fine and Yao in \cite{fine_2018}. 
\\
\\
\textbf{{Notations and conventions:}}
We give a brief explanations of the notation and conventions that are used regularly throughout this paper. The Einstein convention is used. A multi-index in an exponent is used to mean the wedge product. For example, given three $1$-form $e^i, e^j, e^k$, 
    \[e^{ijk}=e^i \wedge e^j \wedge e^k.\] 
Let $dt^j$ be $1$-forms on $\mathbb{T}^3$. The notation $\hat{dt^j}$ is used to express the $2$-form on $\mathbb{T}^3$ that omits $dt^j$. That is 
    \[\hat{dt^j} = dt^i \wedge dt^k\]
    where $i,j,k$ are cyclically permuted $i \to j \to k$. The notation $\varepsilon_{ijk}$ is the permutation symbol on $\{1,2,3\}$.  
\\
\\
\textbf{Acknowledgment:} The results of this paper comprise a portion of the author's doctoral thesis in the Department of Pure Mathematics at the University of Waterloo in 2026.

\section{Geometric structures} \label{Section Geometric Structures}
%----------------------------------------------------------------------
In this section we explain the different structures needed to understand the results of this paper. For this section we let $X^4$ be any oriented manifold of dimension four. We fix a volume form $\mu$ on $X^4$. 
\begin{Def}\label{positivetriple}
    Let $X^4$ be an oriented manifold of dimension four, let $\mu$ be a non-vanishing top form on $X^4$. Let $\underline{\omega} = \l \omega_1,\omega_2,\omega_3\r$ be a triple of $2$-forms on $X^4$. Let $Q$ be the symmetric $3 \times 3$ matrix valued function such that
    \begin{equation}\label{eq:B1}
        \omega_i \wedge \omega_j = 2  Q_{ij}  \mu.
    \end{equation}
    We say the triple $\underline{\omega}$ is a positive triple if $Q$ is positive definite at all points. 
\end{Def}
\begin{remark}
    When $Q$ is positive definite, the $2$-forms $\omega_i$ are non-degenerate as
\[\omega_i^2 = 2 Q_{ii}   \mu \neq 0\]
  which is non-vanishing since $Q_{ii} >0 $ by positive definiteness.
\end{remark}
 \noindent We also note that the matrix $Q$ depends on $\underline{\omega}$ but also on $\mu$. One can see that as $\mu$ changes, $Q$ also changes. There is a particular choice of volume form $\mu$ so that $\det(Q) = 1$, this detail is explained in lemma \ref{relation between volume form and reference form}. This is the volume form that is chosen in \cite{fine_2018}. Furthermore, the positive triple $\underline{\omega}$ gives rise to a Riemannian metric $g_{\underline{\omega}}$ defined by 
\begin{equation}\label{metric}
    g_{\underline{\omega}} (u, v) \mu_{\underline{\omega}} = \frac{1}{6} \varepsilon_{ijk} \l u \iprod \omega_i\r \wedge \l v \iprod \omega_j\r \wedge \omega_k.
\end{equation}
It is important to note that $\mu_{\underline{\omega}}$ is the Riemannian volume form of $g_{\underline{\omega}}$ and is different from our reference form $\mu$ on $X^4$. However $\mu_{\underline{\omega}}$ can be expressed in terms of $Q$ and $\mu$, as we show in the following lemma. We want to make clear that $Q$ scales depending on $\mu$ but $\mu_{\underline{\omega}}$ is always expressed in the same way as in lemma \ref{relation between volume form and reference form}.  Furthermore, it is not obvious that \eqref{metric} gives a Riemannian metric, this is due to the positive definiteness of $Q$. We refer the reader to \cite{fine_2018} and \cite{fine2020reporthypersymplecticflow} for more details regarding this metric construction.
\begin{remark}
    Each $\omega_i$ in the positive triple is self-dual with respect to the metric and orientation given in \eqref{metric} \cite{fine_2018}.
\end{remark}

\begin{lemma}\label{lemma2} Let $S$ be a $3 \times 3$ matrices. Let $\det(S)$ be the determinant of $S$. Then 
\begin{equation}\label{det(s) Lemma}
\begin{split}
     \sum_{i,j,k} \varepsilon_{ijk}  S_{ip}  S_{jq}  S_{kl} &= \l \det(S)\r  \varepsilon_{pql}.
\end{split}
\end{equation}    
\end{lemma}

\begin{proof}
    The identity $\varepsilon_{ijk}  S_{ip}  S_{jq}  S_{kl}$ is skew in $p,q,$ and $l$ thus we can write the expression as being equal to $\lambda \varepsilon_{pql}$ for some scalar $\lambda$. That is,
    \[\sum_{i,j,k}\varepsilon_{ijk}  S_{ip}  S_{jq}  S_{kl} = \lambda \varepsilon_{pql}.\]
    If we take the above and contract it with $\varepsilon_{pql}$ we obtain 
    \begin{align*}
       6 \det(S) = \sum_{p,q,l} \sum_{i,j,k} \varepsilon_{pql} \varepsilon_{ijk}  S_{ip}  S_{jq}  S_{kl} &= \sum_{p,q,l} \lambda \varepsilon_{pql} \varepsilon_{pql} = 6 \lambda . 
    \end{align*}
    We substitute $\lambda = \det(S)$ and we obtain the desired result.
 \end{proof}

\begin{lemma} \label{relation between volume form and reference form}
Let $\underline{\omega}$ be a positive triple on $X^4$. Let $\mu$ be the reference volume form for $\underline{\omega}$ as in definition \ref{positivetriple}. Let $g_{\underline{\omega}} \otimes \mu_{\underline{\omega}}$ be defined as in \eqref{metric}. Then $\mu_{\underline{\omega}} = \l \det(Q)\r^{\frac{1}{3}} \mu$.
\end{lemma}

\begin{proof}
To begin we emphasize that we are doing these computations at some fixed point in $X^4$ as this result and proof is entirely linear algebra. Let $\underline{\omega}$ be a positive triple. Let $\mu$ and $Q$ be such that $2Q_{ij} \mu = \omega_i \wedge \omega_j$. Since $Q$ is a symmetric matrix, there exists an invertible $3 \times 3$ matrix $P$ such that if $\theta_1,\theta_2,\theta_3$ are defined by 
    \[\theta_i = P_{il} \omega_l \quad \text{then} \quad \theta_i \wedge \theta_j= 2 \delta_{ij}\mu. \]
    We note that $\underline{\theta} = \l \theta_1, \theta_2, \theta_3 \r$ is still a positive triple whose matrix $Q_{\underline{\theta}} = \mathbb{I}_{3 \times 3}$.  When we expand the expression for the wedge product above, we obtain
    \[P_{ik} \omega_k \wedge P_{jl} \omega_l = P_{ik} P_{jl} \omega_k \wedge \omega_l = P_{ik}P_{jl} 2 Q_{kl} \mu, \]
    \[\implies P_{ik} Q_{kl}P_{jl} = \delta_{ij}.\]
    This tells us that $PQP^{T} = \mathbb{I}_{3 \times 3}$. Let $M = P^{-1}$. By taking the determinant of the expression above, we have
    \[\det(P) \det(Q) \det(P) = 1 \quad \text{and hence} \quad \det(Q) = (\det(M))^{2}. \]
    Since we let $M = P^{-1}$, we can say that $\omega_i = M_{il} \theta_l$. Now let $g_{\underline{\theta}}$ be the metric coming from the triple $\underline{\theta}$ defined as in \eqref{metric}. One can show that there locally exists a coframe $e^0,e^1,e^2,e^3$ that is orthonormal with respect to $g_{\underline{\theta}}$ so that $\theta_i = e^0 \wedge e^i + \frac{1}{2} \varepsilon_{ijk} e^j \wedge e^k$ and $\mu = e^0\wedge e^1 \wedge e^2 \wedge e^3$. Let $e_i$ be the dual frame to $e^i$. We use the formula for the metric \eqref{metric} to say 
    \begin{equation}\label{proof 1 identity}
        \frac{1}{6} \varepsilon_{ijk} \l e_a \iprod \theta_i \r \wedge \l e_b \iprod \theta_j\r \wedge \theta_k = \delta_{ab} \mu.
    \end{equation}
    We know that given a basis, $\mu_{\underline{\omega}} = \sqrt{\det(g_{\underline{\omega}})} \mu$ where $g_{\underline{\omega}}$ is the metric determined by the triple $\underline{\omega}$. We wish to determine $\sqrt{\det(g_{\underline{\omega}})}$. To do this we know that $g_{\underline{\omega}}$ can be expressed as a $4 \times 4$ matrix, thus if we let $R$ be a $4 \times 4$ matrix such that $R_{ab} \mu = g_{\underline{\omega}}(e_a,e_b) \mu_{\underline{\omega}}$ and use  lemma \ref{lemma2} and \eqref{proof 1 identity} then
    \begin{align*}
        R_{ab} \mu = g_{\underline{\omega}}(e_a,e_b) \mu_{\underline{\omega}} &= \frac{1}{6} \varepsilon_{ijk} \l e_a \iprod \omega_i \r \wedge \l e_b \iprod \omega_j\r \wedge \omega_k \\
        &= \frac{1}{6} \varepsilon_{ijk}M_{ip}M_{jq}M_{kl} \l e_a \iprod \theta_p \r \wedge \l e_b \iprod \theta_q\r \wedge \theta_l \\
        &= \frac{1}{6} \det(M) \varepsilon_{pql} \l e_a \iprod \theta_p \r \wedge \l e_b \iprod \theta_q\r \wedge \theta_l \\
        &= \det(M) \delta_{ab} \mu.
    \end{align*}
 The above computation determines that $R = \det(M) \mathbb{I}_{4 \times 4}$. We use this computation and the fact that $\mu_{\underline{\omega}} = \sqrt{\det(g_{\underline{\omega}})} \mu$ to obtain
   \begin{equation}\label{proof 1 identity 2}
       \begin{split}
           R = g_{\underline{\omega}} \sqrt{\det(g_{\underline{\omega}})} = \det(M) \mathbb{I}_{4 \times 4}.
       \end{split}
   \end{equation}
   We apply the determinant to \eqref{proof 1 identity 2} to obtain
   \begin{align*}
       \det(g_{\underline{\omega}} \sqrt{\det(g_{\underline{\omega}})}) &= \det(\det(M) \mathbb{I}_{4 \times 4}) \\
       \det(g_{\underline{\omega}})^3 &= \det(M)^4 \\
       \det(g_{\underline{\omega}})^3 &= \det(Q)^2  \quad \text{since} \quad \det(Q) = \det(M)^2\\
       \sqrt{\det(g_{\underline{\omega}})} &= \det(Q)^{\frac{1}{3}}.
   \end{align*}
  We substitute this last equation back into the formula for $\mu_{\underline{\omega}}$ in terms of $\mu$ and obtain our desired result $\mu_{\underline{\omega}} = \l \det(Q)\r^{\frac{1}{3}} \mu$.
\end{proof}
\noindent This lemma tells us that when $\det(Q) =1$ the volume form $\mu_{\underline{\omega}}$ is equal to the reference volume form $\mu$. In particular, it tells us how to choose $\mu$ so that the determinant of $Q$ is always $1$. If we first construct $\mu_{\underline{\omega}}$ using the formula for the metric \eqref{metric} and choose $\mu = \mu_{\underline{\omega}}$ this ensures that $\det(Q) = 1$. This is important in particular because as we flow our geometric structures in time $\mu_{\underline{\omega}}$, $\underline{\omega}$, and $Q$ will depend on time but the determinant of $Q$ at any fixed point in time will always be $1$.

\noindent The relation \eqref{eq:B1} and the Riemannian metric \eqref{metric} give us three potential structures on $X^4$. The first happens when $Q$ is the identity matrix. In this case the triple $\underline{\omega}$ forms an $\text{SU}(2)$ structure, see \cite{Fowdar_2025}. The second is a hypersymplectic structure, this is the case where the forms $\omega_i$ are symplectic, meaning they are closed, that is $d\omega_i = 0$. In this case $Q$ need not necessarily be the identity matrix, see \cite{fine_2018}. The last is a hyperk\"ahler structure. In this case the triple forms an $\text{SU}(2)$ structure and a hypersymplectic structure, meaning the forms $\underline{\omega}$ are symplectic (that is they are closed) and the matrix $Q$ is the identity, see \cite{fine2024convergencehypersymplecticflowt4}, \cite{donaldson2016adiabaticlimitscoassociativekovalevlefschetz}. We give the full definition of hypersymplectic and hyperk\"ahler below.  

\begin{Def} \label{Hyperkahler definition}
Let $(M,g)$ be a Riemannian manifold equipped with three complex structures $J_1, J_2, J_3 \colon TM \to TM$ such that 
\begin{equation*}
    J_1^2 = J_2^2 = J_3^2 = J_1J_2J_3 = -1.
\end{equation*}
If $J_1,J_2,J_3$ are parallel then $(M,g, J_1, J_2, J_3)$ is a hyperk\"ahler manifold. Furthermore, we have $\omega_i(x,y) = - g(x, J_iy)$ for the K\"ahler forms. There is one for each complex structure $J_1,J_2,J_3$. These $2$-forms are closed and non-degenerate. 
\end{Def}
\begin{remark}
    When a positive triple $\underline{\omega}$ is hyperk\"ahler, where the forms $\omega_i$ are given as in the definition above, the matrix $Q$ that comes from the positive triple is the identity. We show this below.
\end{remark}
\noindent To verify the remark made above, let $M$ be a $4$-dimensional hyperk\"ahler manifold. Let $e_0, e_1, e_2, e_3$ be orthonormal basis with respect to $g$ as in definition \ref{Hyperkahler definition}, let $e^i$ for $i = 0,1,2,3$ be its dual, and let $\mu = e^{0123}$. Let $J_i$ be the three complex parallel structures on $M$ so that $\omega_i (x,y) = -g(x, J_i y)$ and $J_i(e_0) = e_i$, then we have the following relations: 
\begin{align*}
    \omega_i(e_0,e_i) &= -g(e_0, J_i(e_i)) = 1, \\
    \omega_i(e_0,e_j) &= -g(e_0,J_i(e_j)) = -g(e_0,e_k) = 0, \\
    \omega_i (e_j,e_i) &= -g(e_j,J_i(e_i)) = g(e_j,e_0) = 0,\\
    \omega_i (e_j,e_k) &= -g(e_j, J_i(e_k)) = g(e_j,e_j) = 1.
\end{align*}
This allows us to say that in this basis $\omega_i = e^0\wedge e^i + e^j \wedge e^k$. It is then easy to see that when the triple is written in this form the relation $\omega_i \wedge \omega_j = 2 \delta_{ij} \mu$ is satisfied. Since $Q= \mathbb{I}_{3 \times 3}$ and $\det(Q)=1$ we have $\mu = \mu_{\underline{\omega}}$. Furthermore, one can show that $g_{\underline{\omega}}$ is indeed the hyperkahler $g$ from definition \ref{Hyperkahler definition}.

\begin{remark}
    When the $2$-forms $\omega_i$ are all closed, the complex structures they define are integrable, see \cite{HitchinHK}, \cite{MR1798605}. Since the intrinsic torsion of an $\text{SU}(2)$ structure vanishes if and only if each $2$-form $\omega_i$ is closed, see \cite{Fowdar_2025}, one can think of a hyperk\"ahler structure in dimension $4$ as a torsion free $\text{SU}(2)$ structure.   
\end{remark}
\begin{Def}\label{Definitions HS 1}
Let $X^4$ be a $4$ dimensional oriented smooth manifold. A hypersymplectic structure on $X^4$ is a positive triple of symplectic forms $\underline{\omega} = (\omega_1, \omega_2, \omega_3)$.
\end{Def}
 \noindent We can always take a hypersymplectic triple $\underline{\omega}$ and diagonalize so that $Q = \mathbb{I}_{3 \times 3}$.  This gives us a $\text{SU}(2)$ structure at the cost of the closed-ness of the triple. When the matrix $Q$ coming from the hypersymplectic triple is constant, we can apply a linear transformation to our triple $\underline{\omega}$ so that it is now a triple $\Tilde{\underline{\omega}}$ for which $Q$ is now the identity matrix, and $\tilde{\underline{\omega}}$ is closed thus hyperk\"ahler. Then $\underline{\omega}$ induces the same metric as $\tilde{\underline{\omega}}$ up to a constant factor. We show this explicitly now. Let $\underline{\omega}$ be a hypersymplectic triple where $Q$ given as in definition \ref{Definitions HS 1} is constant. Then there exists some constant matrix $A$ that diagonalizes $Q$, meaning we can define another triple $\tilde{\underline{\omega}}$ by $\tilde{\omega}_i = A_{ik} \omega_k$ where $\tilde{\omega_i} \wedge \tilde{\omega_j} = 2\delta_{ij} \mu$. Since $A$ is constant the triple $\tilde{\underline{\omega}}$ is still closed since 
\[d\tilde{\omega}_i = dA_{ik} \wedge \omega_k + A_{ik}  d \omega_k = 0.\]
Furthermore, let the Riemannian metric as defined in \eqref{metric} derived using the triple $\tilde{\underline{\omega}}$ be $g_{\tilde{\underline{\omega}}}$ and let the metric defined using $\underline{\omega}$ be denoted by $g_{\underline{\omega}}$. Then the metrics are constant multiples of one another: 
\begin{equation}\label{determinant and metric relation}
    g_{\tilde{\underline{\omega}}} = \det(A) g_{\underline{\omega}}.
\end{equation}
This computation can be easily seen from some of the computations done in the proof of lemma \ref{relation between volume form and reference form}. Since $A$ is a constant matrix, $\det(A)$ is also constant. Therefore the metric $g_{\underline{\omega}}$ is still Ricci-flat. We conclude that when $Q$ is constant for a hypersymplectic triple we are able transform to a hyperk\"ahler triple whose metric is the same up to a constant factor. 

%---------------------------------------------------------------------------------------------------------------------------

%This is the chapter where we consructut the positive triple and the different closed and co closed structures 

\section{$G_2$ structures and the hypersymplectic flow} \label{Chaper Laplacian Co-flow}
In Section \ref{Section Geometric Structures} we introduced a positive triple $\underline{\omega}$. In this section we begin by introducing a dual positive triple. Using a closed positive triple (a hypersymplectic structure) we can construct either a closed $G_2$ structure or a coclosed $G_2$ structure. Whichever one we choose to construct, the dual triple will show up in the other form. If we make the assumptions that the dual triple is also closed then we obtain a torsion free $G_2$ structure. In \cite{fine_2018}, the authors used a hypersymplectic triple to define a $3$-form $\varphi$ that gave a closed $G_2$ structure on $\mathbb{T}^3 \times X^4$. We follow this analogy and construct a coclosed $G_2$ structure on $\mathbb{T}^3 \times X^4$ by defining a $4$-form using a hypersymplectic triple in a similar manner to Fine-Yao in \cite{fine_2018} and the dual to this triple will show up in the construction of the $3$-form. We find that this $G_2$ structure is invariant under $\mathbb{T}^3$.  Then we run the $G_2$-Laplacian coflow that is defined in \cite{grigorian2018flowscoclosedg2structures} on our $4$-form. In our particular setting the Laplacian coflow and modified coflow agree. We verify that solutions to this flow are also invariant under $\mathbb{T}^3$.  We find that it descends to a flow on $X^4$ of the dual hypersymplectic triple which is precisely the hypersymplectic flow defined in $\cite{fine_2018}$. 

\subsection{Dual positive triple and a coclosed $G_2$ structure}\label{Dual positive triple and a coclosed $G_2$ structure}

Given some positive triple $\underline{\omega}$ as in definition \ref{positivetriple} we can define a dual positive triple. If $Q$ is the symmetric positive definite matrix that defines the positive triple $\underline{\omega}$, then its inverse matrix $Q^{-1}$ is also symmetric and positive definite. We use this to define a new positive triple $\underline{\sigma}$ called the dual positive triple. 

\begin{lemma}
     Let $X^4$ be an oriented manifold of dimension four, let $\underline{\omega} = (\omega_1, \omega_2, \omega_3)$ be a positive triple on $X^4$.  
    Let $Q$ be the matrix determined by $\underline{\omega}$ and the volume form $\mu_{\underline{\omega}}$ so that $\det(Q) = 1$, as in lemma \ref{relation between volume form and reference form}. Then $\underline{\sigma} = (\sigma_1, \sigma_2, \sigma_3)$ given by
    \begin{equation}\label{dualform}
        \sigma_i = \l Q^{-1}\r_{ik} \omega_k = Q^{ik} \omega_k
    \end{equation}
    is another a positive triple on $X^4$.
\end{lemma}

\begin{proof}
    We remark that since $Q$ is symmetric and positive definite, so is $Q^{-1}$. Thus if $\omega_i \wedge \omega_j = 2 Q_{ij} \mu_{\underline{\omega}}$ then the triple $\underline{\sigma}$ given as in \eqref{dualform} satisfies
\begin{equation}\label{dual form wedge}
\begin{split}
    \sigma_i \wedge \sigma_j = \l Q^{ik} \omega_k\r \wedge \l Q^{jl} \omega_l\r &= Q^{ik}Q^{jl} \omega_k \wedge \omega_l \\
    &= 2 Q^{ik}Q^{jl}Q_{kl}  \mu_{\underline{\omega}} \\
    &= 2 Q^{ik} \delta_{jk}  \mu_{\underline{\omega}} \\
    &= 2 Q^{ij}  \mu_{\underline{\omega}}.
\end{split}
\end{equation}
Since $Q^{-1}$ is positive definite this along with the computation \eqref{dual form wedge} confirms that the triple $\underline{\sigma}$ given as in \eqref{dualform} defines a positive triple on $X^4$.
\end{proof}

\begin{remark}
     We call the positive triple \eqref{dualform} the dual positive triple to $\underline{\omega}$.
\end{remark}

\noindent If we require the $2$-forms $\sigma_i$ to be closed, then the triple $\sigma_i$ is a hypersymplectic triple. It is important to note that if the positive triple $\underline{\omega}$ is closed it is not necessarily true that the dual triple $\underline{\sigma}$ is also closed. Furthermore we note that the computation \eqref{dual form wedge}, alongside the way we have chosen $ \mu_{\underline{\omega}}$ so that $\det(Q)=1$ tells us that $ \mu_{\underline{\omega}} =  \mu_{\underline{\sigma}}$. 

 \noindent In \cite{fine_2018} the authors use a hypersymplectic triple $\underline{\omega}$ on $X^4$ to define a $G_2$ form $\varphi = dt^{123} - dt^i \wedge \omega_i$ on $\mathbb{T}^3 \times X^4$ that is closed. This form induces the following metric 
 \begin{equation}\label{Fine Yao G_2 metric}
     g_{\varphi} = Q_{ij} dt^i dt^j + g_{\underline{\omega}}
 \end{equation}
 where $g_{\underline{\omega}}$ is the metric on $X^4$ as in \eqref{metric}, $Q$ is the matrix from the triple $\underline{\omega}$ and $ \mu_{\underline{\omega}}$, and $t^i$ for $i = 1,2,3$ are the coordinates on $\mathbb{T}^3$. We now show that if we use the dual positive triple, the metric we obtain via \eqref{metric} is the same. Let $\underline{\omega}$ be a positive triple and let $\underline{\sigma}$ be the dual triple defined as in \eqref{dualform}. Using the formula for the metric \eqref{metric} we obtain 
\begin{align*}
    g_{\underline{\sigma}}(u,v) \mu_{\underline{\sigma}} &= \frac{1}{6} \varepsilon_{ijk}\l u \iprod \sigma_i\r \wedge \l v \iprod \sigma_j \r \wedge \sigma_k \\
    &= \frac{1}{6} \varepsilon_{ijk}\l u \iprod Q^{im} \omega_m\r \wedge \l v \iprod Q^{jn} \omega_n\r \wedge Q^{kl} \omega_l \\
    &= \frac{1}{6} \varepsilon_{ijk} Q^{im} Q^{jn}Q^{kl} \l u \iprod \omega_m\r \wedge \l v \iprod \omega_n\r \wedge \omega_l \\
    &= \frac{1}{6} \det (Q^{-1}) \varepsilon_{mnl} \l u \iprod \omega_m\r \wedge \l v \iprod \omega_n\r \wedge \omega_l \\
    &=  g_{\underline{\omega}}(u,v) \mu_{\underline{\omega}}
\end{align*}
 where the last line uses the fact that we defined $Q$ using the volume form $ \mu_{\underline{\omega}}$ so that $\det(Q) = 1$. We will use a similar analogy to \cite{fine_2018} but instead we define the $G_2$ $4$-form $\psi$ using the dual hypersymplectic structure. Let $\underline{\omega}$ be a positive triple with matrix $Q$ and let $\underline{\sigma}$ be its dual positive triple. Suppose that $d\underline{\sigma} = 0$. Then we can define a closed $4$-form $\psi$ on $\mathbb{T}^3 \times X^4$ as follows
 \begin{equation}\label{psi}
     \psi =  \mu_{\underline{\sigma}} - dt^{12} \wedge \sigma_3 - dt^{31} \wedge \sigma_2 - dt^{23} \wedge \sigma_1.
 \end{equation}
 We claim that this defines a coclosed $G_2$ structure on $\mathbb{T}^3 \times X^4$ where $g_{\psi} = g_{\varphi}$ where $g_{\varphi}$ is as in \eqref{Fine Yao G_2 metric}. The way to see this is that using the Hodge star $*_7$ coming from the metric $g_{\varphi}$, we compute $*_7 \psi$ and obtain the precise $3$-form $\varphi$ defined in \cite{fine_2018} which induces the metric $g_{\varphi}$. Since every $G_2$ $4$-form induces a unique metric we obtain that $g_{\varphi} = g_{\psi}$. To this end, the closed $4$-form in \eqref{psi} gives a coclosed $G_2$ structure on $X^4$. Before we show that  $*_7 \psi$ is indeed $\varphi = dt^{123} - dt^i \wedge \omega_i$, we introduce some technical lemmas regarding the Hodge star computations. 

\begin{lemma}\label{hodge star decomp lemma}
   Let $(N^3, g_3)$ an $(X^4, g_4)$ be two Riemannian manifolds. Let $M^7 = N^3 \times X^4$ be a Riemannian manifold where at each point the metric splits as $g_7 = g_3 \oplus g_4$. Let $\alpha \in \Omega^k(N^3)$ and $\beta \in \Omega^l(X^4)$. Then
   \[*_7(\alpha \wedge \beta) = (-1)^{l(k+1)}(*_3 \alpha) \wedge (*_4 \beta).\]
\end{lemma}
\begin{proof}
 Let $\vol_7$, $\vol_3$, and $\vol_4$ be the volume forms on $M^7$, $N^3$, and $X^4$ respectively thus $\vol_7 = \vol_3 \wedge \vol_4$. Then by definition of $*_7$ we have 
\begin{align*}
    \alpha \wedge \beta \wedge  *_7(\alpha \wedge \beta) &= |\alpha \wedge \beta|^2_{g_7}\vol_7 \\
    &= |\alpha|^2_{g_3} |\beta|^2_{g_4}\vol_3 \wedge \vol_4\\
    &= \alpha \wedge *_3 \alpha \wedge \beta \wedge *_4 \beta \\
    &= (-1)^{l \l 3-k\r} \alpha \wedge \beta \wedge *_3 \alpha \wedge *_4 \beta \\
    &= (-1)^{l \l k+1\r} \alpha \wedge \beta \wedge *_3 \alpha \wedge *_4 \beta 
\end{align*}
 Therefore we can conclude that $*_7(\alpha \wedge \beta) = (-1)^{l(k+1)}(*_3 \alpha) \wedge (*_4 \beta)$.
\end{proof}

\begin{lemma}
     Let $X^4$ is a  $4$-manifold with a positive triple $\underline{\omega}$. Let $M^7 = \mathbb{T}^3 \times X^4$. Let $Q$ be the matrix determined by the positive triple and the reference volume form chosen so that $\det(Q) = 1$, as in definition \ref{positivetriple}. Let $t^i$ be the flat coordinates on $\mathbb{T}^3$. Let $\varphi = dt^{123} - dt^i \wedge \omega_i$ be $G_2$ structure on $M^7$ with induced metric $g_7 = Q_{ij} dt^idt^j + g_{\underline{\omega}}$. Then 
\begin{equation}\label{hodge star of 1 forms}
    \begin{split}
        *_3dt^1 &= Q^{13}dt^{12} + Q^{11}dt^{23} + Q^{21}dt^{31}, \\
        *_3dt^2 &= Q^{23}dt^{12} + Q^{21}dt^{23} + Q^{22}dt^{31},\\
        *_3dt^3 &= Q^{33}dt^{12} + Q^{31}dt^{23} + Q^{32}dt^{31},
    \end{split}
\end{equation}
    and 
    \begin{equation}\label{hodge star of 2 forms}
\begin{split}
     *_3 dt^{31} &= Q_{12} dt^1 + Q_{22}dt^2 + Q_{32} dt^3, \\
    *_3 dt^{12} &= Q_{31}dt^1 + Q_{32}dt^2 + Q_{33}dt^3, \\
    *_3 dt^{23} &= Q_{11}dt^1 + Q_{12}dt^2 +Q_{13}dt^3. 
\end{split}
\end{equation} 
\end{lemma}

\begin{proof}
    Let us start with proving \eqref{hodge star of 1 forms}. To begin we know from lemma \ref{hodge star decomp lemma} that $*_7 dt^i = *_3 dt^i \wedge \vol_4$. Let $\alpha \in \Omega^1(M^7)$, so  
    \begin{align*}
        \alpha \wedge *_7 dt^i &= \alpha \wedge *_3 dt^i \wedge \vol_4 \\
       \implies g_7(\alpha, dt^i) \vol_3 \wedge \vol_4 &= \alpha \wedge *_3 dt^i \wedge \vol_4 \\
       \implies g_7(\alpha,dt^i) dt^{123} &= \alpha \wedge *_3 dt^i
    \end{align*}
    which tells us that $*_3 dt^i$ can be computed using the metric $g_7$. Furthermore since $g_7$ has no mixed pieces we need only consider $\alpha \in \Omega^1(\mathbb{T}^3)$. Note that a similar argument can be made for $*_3 dt^{ij}$ and thus in that case we need only consider $\beta \in \Omega^2(\mathbb{T}^3)$. Returning to $*_3dt^i$, we only show the details for $*_3dt^1$ as the remaining computations are done similarly. Let the entries of $Q^{-1}$ be denoted by $Q^{ij}$. Let $*_3 dt^1 = A dt^{12} + B dt^{23} + Cdt^{31}$ where $A,B,C$ are smooth functions on $M^7$. We compute the following 
    \begin{align*}
        dt^1 \wedge *_3 dt^1 &= g_7(dt^1, dt^1) dt^{123} \\
        B dt^{123} &= Q^{11} dt^{123} \\
        \implies B &= Q^{11},\\ 
        dt^2 \wedge *_3 dt^1 &= g_7(dt^2,dt^1) dt^{123} \\
        C dt^{123} &= Q^{21} dt^{123}\\
        \implies C &= Q^{21},\\
        dt^{3} \wedge *_3 dt^1 &= g_7(dt^3, dt^1) dt^{123} \\
        \implies A &= Q^{31}.
    \end{align*}
   Combining these together we obtain that  $*_3dt^1 = Q^{13}dt^{12} + Q^{11}dt^{23} + Q^{21}dt^{31}$. We can repeat this same process for $*_3dt^2$ and $*_3 dt^3$ and obtain the other desired identities in $\eqref{hodge star of 1 forms}$. Now we compute the identities \eqref{hodge star of 2 forms}, and by similar logic to the case of $1$-forms, the computation of $*_3dt^{ij}$ can be done using only $g_7$ and we need only consider $2$-forms on $\mathbb{T}^3$. Recall that $\det(Q) = \det(Q^{-1}) = 1$ and therefore we can use the following formula for $Q = \frac{\mathrm{adj}(Q^{-1})}{\det(Q^{-1})} = \mathrm{adj}(Q^{-1})$. Using this identity and the fact that $Q$ is symmetric we can compute $*_3dt^{ij}$. We will only do this computation for $*_3dt^{31}$ as the remaining computations are done similarly. Let $*_3 dt^{31} = A dt^1 + B dt^2 + C dt^3$ then we obtain the following: 
   \begin{align*}
       dt^{12} \wedge *_3 dt^{31} &= g_7(dt^{12}, dt^{31}) dt^{123} \\
       C dt^{123} & = - \l Q^{23} Q^{11} - Q^{13} Q^{21} \r dt^{123} \\
       \implies C &= Q_{32}, \\
       dt^{31} \wedge *_3 dt^{31} &= g_7(dt^{31}, dt^{31}) dt^{123} \\
       B dt^{123} &= (Q^{33}Q^{11} - (Q^{13})^2) dt^{123}\\ 
       \implies B &= Q_{22}, \\
       dt^{23} \wedge *_3 dt^{31} &= g_7(dt^{23}, dt^{31}) dt^{123}\\
       A dt^{123} &= - (Q^{33}Q^{21} - Q^{23} Q^{31})dt^{123} \\
       \implies A &= Q_{12}. 
   \end{align*}
   Combining these together we obtain the desired result $*_3 dt^{31} = Q_{12} dt^1 + Q_{22}dt^2 + Q_{32} dt^3$.
\end{proof}

\noindent We now show that $*_7 \psi$ is indeed $\varphi = dt^{123} - dt^i \wedge \omega_i$. Using the fact that $\sigma_i$ are self dual and lemma \ref{hodge star decomp lemma} and \eqref{hodge star of 2 forms}  we compute
 \begin{equation*}
     \begin{split}
         \varphi = *_7 \psi &= dt^{123} - *_3dt^{12} \wedge \sigma_3 - *_3dt^{31} \wedge \sigma_2 - *_3dt^{23} \wedge \sigma_1 \\
         &= dt^{123} - \l Q_{3k} dt^k \r \wedge \sigma_3 - \l Q_{2k} dt^k\r \wedge \sigma_2 - \l Q_{1k} dt^k\r \wedge \sigma_1 \\
         &=  dt^{123} - \l Q_{3k} dt^k \r \wedge \l Q^{3i} \omega_i \r - \l Q_{2k} dt^k\r \wedge \l Q^{2i} \omega_i \r - \l Q_{1k} dt^k\r \wedge \l Q^{1i} \omega_i\r \\
         &= dt^{123} - dt^i \wedge \omega_i.
     \end{split}
 \end{equation*}
 We reiterate that $d\underline{\omega} =0$ gives a closed $G_2$ structure and $d\underline{\sigma} =0$ gives a coclosed $G_2$ structure. However if both the positive triple $\underline{\omega}$ and its dual positive triple $\underline{\sigma}$ are closed (both are hypersymplectic) then the $G_2$ structure they induce is both closed and coclosed and therefore torsion free.

 %------------------------------------------------------------------------------------------------------------------------

 %This is the sectio where we show T^3 invariance and show how the hypersymplectic flow descends from the G2 LAplacian flow 

\subsection{The hypersymplectic flow from the Laplacian coflow}\label{The hypersymplectic flow from the Laplacian coflow}
In the previous section we defined a coclosed $G_2$ structure on $\mathbb{T}^3 \times X^4$ coming from a hypersymplectic structure in an analogous way to what was done \cite{fine_2018}. In this section we compute the flow of the $4$-form $\psi$ from \eqref{psi} under the modified Laplacian coflow from \cite{Grigorian_2013}. To begin we note that in this section $X^4$ is compact. We first show that the $G_2$ structure given by $\psi$ as in \eqref{psi} and $\varphi$ as in $\eqref{G2 form}$ the Laplacian coflow and modified coflow agree. Let us recall the forms 
\begin{equation}\label{3 and 4 form}
    \begin{split}
         \varphi &= dt^{123} - dt^1 \wedge \omega_1 - dt^2 \wedge \omega_2 - dt^3 \wedge \omega_3, \\
    \psi &=  \mu_{\underline{\sigma}} - dt^{12} \wedge \sigma_3 - dt^{31} \wedge \sigma_2 - dt^{23} \wedge \sigma_1,
    \end{split}
\end{equation}
where $\underline{\omega}$ is a positive triple with dual triple $\underline{\sigma}$ with the further imposition that $d \underline{\sigma} = 0$ giving a coclosed $G_2$ structure as was defined in the previous section. The Laplacian coflow for coclosed $G_2$ structures is given by 
\begin{equation}\label{flow of psi 0}
    \frac{\partial \psi}{\partial t} = dd^{*_7}\psi = d *_7 d *_7 \psi
\end{equation}
where $d^{*_n} = (-1)^{nk+n+1} *_nd *_n$ on $k$-forms over a manifold of dimension $n$. Recall the modified Laplacian coflow for coclosed $G_2$ structures is given by 
\begin{equation} \label{modified coflow}
     \frac{\partial \psi}{\partial t} =  dd^* \psi + 2d\l \l A- \mathrm{Tr}T_{\psi} \r *_{\psi}\psi\r
\end{equation}
where $A$ is a choice of constant and $T$ is the full torsion tensor of the $G_2$ structure. Before we can compute how the hypersymplectic triple $\underline{\sigma}$ flows under this flow we first need to show that solutions to this flow remain $\mathbb{T}^3$ invariant for all time. 

\begin{theorem}
    Let $\psi(t)$ be a solution of the modified Laplacian coflow \eqref{modified coflow} on $\mathbb{T}^3 \times X^4$ with initial condition $\psi_0$ given by 
    \begin{equation}\label{initial condition for coflow}
        \psi_0 = \mu_{\underline{\sigma}} - \frac{1}{2}\varepsilon_{ijk} dt^{jk} \wedge \sigma_i
    \end{equation}
    where $\underline{\sigma}$ is hypersymplectic structure. Then $\psi(t)$ is $\mathbb{T}^3$ invariant for as long as it exists. Moreover, $\psi(t)$ remains of the form \eqref{3 and 4 form} for as long as it exists.
\end{theorem}

\begin{proof}
    Let $\psi_0$ be as in \eqref{initial condition for coflow}. We want to show that if $\psi_0$ is $\mathbb{T}^3$ invariant than so is any solution of the modified Laplacian coflow with this initial condition. Let $F$ be any diffeomorphism and given that $\psi$ does not determine an orientation we fix the standard orientation on $\mathbb{T}^3 \times X^4$.  Let $\psi_t$ be a solution for the modified coflow with initial condition $\psi_0$. Then if $F$ is orientation preserving we fix the same orientation on $\mathbb{T}^3 \times X^4$ for $F^*\psi_t$ and if $F$ is orientation reversing we choose the opposite orientation for $F^*\psi_t$, this ensures in the computation below that the pullback of $F$ commutes with the Hodge star $*_{\psi_t}$, that is $*_{F^*_{\psi_t}}(F^* \alpha) = F^*(*_{\psi_t}\alpha)$ for a $k$-form $\alpha$ on $\mathbb{T}^3 \times X^4$.
    Then we compute:
\begin{equation}\label{diffeomorphism invarian coflow}
    \begin{split}
\frac{\partial}{\partial t} F^* \psi_t &= F^* \l \frac{\partial }{\partial t} \psi_t \r \\
&= F^*\l dd^{*_{\psi_t}} \psi_t + 2d\l \l A- \mathrm{Tr}T_{\psi_t} \r *_{\psi_t}\psi_t\r \r \\
&=  dd^{*_{F^*\psi_t}} F^*\psi_t + 2d\l \l A- \mathrm{Tr}T_{F^*\psi_t} \r *_{F^*\psi_t}F^*\psi_t\r.
    \end{split}
\end{equation}
This means that $\psi_t$ and $F^* \psi_t$ both solve the modified coflow with the same initial condition $\psi_0$ and therefore by uniqueness of solutions of the flow we obtain $F^* \psi_t = \psi_t$ for all $t$. Thus $\psi_t$ is $\mathbb{T}^3$-invariant for all time $t$ for which it exists. Next we claim that any solution $\psi(t)$ of the modified coflow has the form 
\begin{equation}\label{4 form for all time}
    \psi(t) = \mu_{\underline{\sigma}}(0) - \frac{1}{2}\varepsilon_{ijk} dt^{jk} \wedge \sigma_i(t)
\end{equation}
where $\mu_{\underline{\sigma}}(0)$ is the volume form that depends on the initial coclosed $G_2$ structure coming from the hypersymplectic triple $\underline{\sigma}$. To see this we let $\psi(t)$ be a solution of the modified coflow with initial condition $\psi_0$. We can write $\psi(t)$ as general as possible as
\begin{equation}\label{4 form general}
    \begin{split}
        \psi(t) = \lambda (t) \mu_{\underline{\sigma}}(0) + dt^i \wedge A_i(t) +\frac{1}{2}\varepsilon_{ijk} dt^{ij} \wedge B_k(t) + dt^{123} \wedge C(t)
    \end{split}
\end{equation}
where $\lambda(t) \in C^{\infty}(X^4)$, $A_i(t) \in \Omega^3(X^4)$, $B_k(t) \in \Omega^2(X^4)$ for all $i,k$ and $C(t) \in \Omega^1(X^4)$.  Since we know that the modified coflow preserves the closedness of the $4$-form we have that all of the forms in \eqref{4 form general} are closed. This tells us that $\lambda(t)$ is constant on $X^4$ and thus only depends on $t$. Now let $\theta: \mathbb{T}^3 \times X^4  \to \mathbb{T}^3 \times X^4$ be the map that takes $(t, p) \mapsto (-t,p) $. We also know from our earlier computation \eqref{diffeomorphism invarian coflow} that $\theta^* \psi_t = \psi_t$, since $\theta^* \psi_0 = \psi_0$. Then
\begin{equation}
    \begin{split}
        \theta^* \psi_t &= \lambda (t) \mu_{\underline{\sigma}}(0) - dt^i \wedge A_i(t) + \frac{1}{2}\varepsilon_{ijk} dt^{ij} \wedge B_k(t) - dt^{123} \wedge C(t) \\
        &= \psi_t
    \end{split}
\end{equation}
which tells that $A_i(t) = 0$ for all $i$ and $C(t) = 0$. Now we claim that $\lambda(t)$ is constant. To see this we compute
\begin{equation}\label{integral of phi}
    \begin{split}
        \int_{X^4} \psi_t &= \int_{X^4} \lambda(t) \mu_{\underline{\sigma}}(0) + \int_{X^4} \frac{1}{2}\varepsilon_{ijk} dt^{ij} \wedge B_k(t) \\
        &= \int_{X^4} \lambda(t) \mu_{\underline{\sigma}}(0) + 0 \\
        &= \lambda (t) \mathrm{vol}_0(X^4).
    \end{split}
\end{equation}
Now we take the time derivative of both sides of \eqref{integral of phi} and we recall that $\psi_t$ is a solution of the modified coflow and thus $\frac{\partial}{\partial t} \psi_t$ is an exact form so its integral will vanish. We obtain
\begin{equation}
    \begin{split}
        0 =\frac{\partial}{\partial t}\l \int_{X^4} \psi_t \r &= \frac{\partial}{\partial t} \l  \lambda (t) \mathrm{vol}_0(X^4)\r = \lambda'(t) \mathrm{vol}_0(X^4).
    \end{split} 
\end{equation}
Hence $\lambda(t)$ is constant and therefore $\lambda(t) = \lambda(0)$ for all $t$. This tells us that if we start with initial condition $\psi_0$ then any solution $\psi(t)$ of the modified coflow will remain of the form \eqref{4 form for all time} for all time for which it exists. In particular this tell us that only the $\sigma_i$'s will depend on time.
\end{proof}

\noindent Now we can return to the modified coflow and show that in our particular setting the coflow \eqref{flow of psi 0} and modified coflow \eqref{modified coflow} agree. In \eqref{modified coflow} we choose $A = 0$. The trace of the full torsion tensor given in  \cite{Karigiannis_2008} is described using the $G_2$ $3$-form as follows
\begin{equation}\label{Trace of torion tensor}
    \Tr \mathrm{T} = \frac{1}{4}*_{7} \l \varphi \wedge d \varphi \r.
\end{equation}
Substituting the $3$-form \eqref{3 and 4 form} into \eqref{Trace of torion tensor} we see that the trace of the torsion tensor vanishes. To see this we compute
\begin{equation*}
    \begin{split}
         \Tr \mathrm{T} &= \frac{1}{4}*_{7} \l \l dt^{123} - dt^i \wedge \omega_i\r \wedge \l dt^j \wedge d\omega_j \r \r \\
         &= \frac{1}{4}*_{7} \l - dt^{ij} \wedge \omega_i \wedge d\omega_j\r \\
         &= 0
    \end{split}
\end{equation*}
where the last equality comes from the fact that $\omega_i \wedge d \omega_j$ is $5$-form completely on $X^4$ and therefore must vanish. As a result for this particular $G_2$ structure the Laplacian coflow and modified coflow agree. Now we flow our coclosed $G_2$ structure \eqref{3 and 4 form} under the Laplacian coflow and we find that it will descend to a flow of the hypersymplectic triple $\underline{\sigma}$ on $X^4$. This flow is precisely the hypersymplectic flow found by Fine-Yao in \cite{fine_2018}.

\begin{theorem} Let $\underline{\sigma}$ be a hypersymplectic structure on $X^4$ and let $\psi$ as in \eqref{3 and 4 form} be the associated coclosed $G_2$ structure on $\mathbb{T}^3 \times X^4$. Then under the $G_2$-Laplacian coflow this descends to a flow of the hypersymplectic triple $\underline{\sigma}$ on $X^4$ given by 
\begin{equation}\label{my hypersymplectic flow}
    \partial_t \underline{\sigma} = d \l Q_{\sigma} d^*\l Q_{\sigma}^{-1} \underline{\sigma} \r\r
\end{equation}
 or on each component 
 \begin{equation*}
      \frac{\partial \sigma_j}{\partial t} =  d \l \l Q_{\sigma}\r_{jk} d^* \l\l Q_{\sigma}\r^{ki} \sigma_i\r \r.
 \end{equation*}
\end{theorem}

\begin{proof}
Let $\underline{\sigma}$ be a hypersymplectic triple. Taking $\psi$ as in \eqref{psi}, the flow of $\psi$ is then 
\begin{equation}\label{flow of psi 1}
    \frac{\partial \psi}{ \partial t} = -dt^{12} \wedge \frac{\partial \sigma_3}{\partial t} - dt^{31}\wedge \frac{\partial \sigma_2}{\partial t} - dt^{23} \wedge \frac{\partial \sigma_1}{\partial t}.
\end{equation}
Recall that in this setting the coflow and modified coflow agree so we expand the right hand side of \eqref{flow of psi 0} using the $\psi$ we defined in \eqref{psi}. To make the following computation easier to read we let $Q$ denote $Q_{\underline{\omega}}$ the matrix associated with the positive triple $\underline{\omega}$. Recall that $\underline{\omega}$ is the dual positive triple to $\underline{\sigma}$, and that $Q_{\underline{\omega}}^{-1} = Q_{\underline{\sigma}}$. Therefore using the fact that the $\sigma_i$'s are self-dual, we have
\begin{align*}
    d*_7d*_7 \psi = d*_7d\varphi &= d*_7 d \l dt^{123} - Q_{ik}dt^k \wedge \sigma_i\r\\ 
    &= d*_7 \l dt^k \wedge d\l Q_{ik} \sigma_i \r\r \\
    &= d \l *_3 dt^k \wedge *_4 d \l Q_{ik} \sigma_i\r\r \\
    &= d \l Q^{kj}\hat{dt^j} \wedge *_4 d \l Q_{ik} \sigma_i\r \r \\
    &=  \hat{dt^j} \wedge d \l Q^{kj} *_4 d *_4 \l Q_{ik} \sigma_i\r  \r \\
    &= - \hat{dt^j} \wedge d \l Q^{kj} d^* \l Q_{ik} \sigma_i\r  \r 
\end{align*}
where $\hat{dt^j}$ denotes the $2$-form on $\mathbb{T}^3$ that omits $dt^j$, for example $\hat{dt^1} = dt^2 \wedge dt^3$. We equate the two computations and obtain a flow of the hypersymplectic structure $\underline{\sigma}$. It is given by
\begin{equation*}
\begin{split}
    \frac{\partial \sigma_j}{\partial t} &=  d \l Q^{kj} d^* \l Q_{ik} \sigma_i\r \r \\
   &= d \l \l Q_{\sigma}\r_{jk} d^* \l \l Q_{\sigma}\r^{ki} \sigma_i\r \r.
\end{split}
\end{equation*}
Since $Q_{\underline{\omega}}^{-1} = Q_{\underline{\sigma}}$ this flow is precisely the hypersymplectic flow of Fine-Yao from \cite{fine_2018}.
\end{proof}

%--------------------------------------------------------------------------------------------------------------------------------            
% The following statement causes the title "References" to be used for the bibliography section:
\renewcommand*{\refname}{References}

\newpage 
\bibliography{main_bibliography.bib}

\cleardoublepage
\phantomsection		

%----------------------------------------------------------------------
\end{document}